\documentclass[reqno,12pt]{amsart}

\parskip = \bigskipamount
\usepackage{amssymb}
\usepackage{amscd}
\usepackage[all]{xy}
\usepackage{bbm}
\usepackage{mathrsfs}
\usepackage{enumerate}
\usepackage{tikz}
\usepackage[margin=1.5in]{geometry}
\usepackage{stmaryrd}
\usepackage{etoolbox}
\usepackage{array}

\newcommand{\N}{\mathbb{N}}
\newcommand{\Z}{\mathbb{Z}}
\newcommand{\Q}{\mathbb{Q}}
\newcommand{\R}{\mathbb{R}}

\newcommand{\inv}{^{-1}}

\newcommand{\eps}{\varepsilon}

\newcommand{\del}{\nabla}
\newcommand{\ind}{\operatorname{ind}}

\newcommand{\bd}{\partial}
\newcommand{\cl}{\overline}

\newcommand{\eval}{\bigg\vert}

\newcommand{\la}{\langle}
\newcommand{\ra}{\rangle}

\newcommand{\supp}{\operatorname{supp}}
\newcommand{\dist}{\operatorname{dist}}
\newcommand{\cc}{\subset\subset}
\renewcommand{\div}{\operatorname{div}}

\newcommand{\grad}{\del}
\newcommand{\vol}{\operatorname{Vol}}

\newcommand{\f}{\colon}
\newcommand{\area}{\operatorname{Area}}

\newcommand{\fla}{\mathcal F}

\newcommand{\ric}{\operatorname{Ric}}

\theoremstyle{plain}
\newtheorem{theorem}{Theorem}
\newtheorem{corollary}[theorem]{Corollary}
\newtheorem{prop}[theorem]{Proposition}
\newtheorem{lem}[theorem]{Lemma}

\theoremstyle{definition}
\newtheorem{defn}[theorem]{Definition}
\newtheorem{rem}[theorem]{Remark}

\begin{document}

\title{CMC Doublings of  Minimal Surfaces via Min-Max}
\author{Liam Mazurowski}
\address{University of Chicago, Department of Mathematics, Chicago, IL 60637}
\email{maz@math.uchicago.edu}

\begin{abstract}
Let $\Sigma^2 \subset M^3$ be a minimal surface of index 0 or 1.  Assume that a neighborhood of $\Sigma$ can be foliated by constant mean curvature (cmc) hypersurfaces.  We use  min-max theory and the catenoid estimate to construct $\eps$-cmc doublings of $\Sigma$ for small $\eps > 0$.  Such cmc doublings were previously constructed for minimal hypersurfaces  $\Sigma^n \subset M^{n+1}$ with $n+1\ge 4$ by Pacard and Sun \cite{PS} using gluing methods.  
\end{abstract}

\maketitle

\section{Introduction} 

Let $M$ be a Riemannian manifold.  A minimal hypersurface $\Sigma \subset M$ is a critical point of the area functional on $M$.  A constant mean curvature (cmc) hypersurface is a critical point of the area functional subject to variations that preserve the enclosed volume.  
A fundamental problem in geometry is to construct minimal and cmc hypersurfaces in a given manifold.  

Min-max methods have long proven to be a powerful tool for constructing minimal surfaces.  In 1981, Pitts \cite{P}, building on work of Almgren \cite{A}, used min-max methods to show that every closed manifold $M^{n+1}$ with $3\le n+1 \le 6$ contains a smooth, embedded minimal hypersurface.  Schoen and Simon \cite{SS} improved this to $3\le n+1\le 7$.  In fact, the work of Schoen and Simon shows that every $M^{n+1}$ with $n+1\ge 3$ contains a minimal hypersurface which is smooth and embedded up to a set of codimension 7. 

In 1982, Yau \cite{Y} conjectured that every closed manifold contains infinitely many minimal surfaces.  Marques and Neves devised a program to prove Yau's conjecture by developing a detailed understanding of the Morse theory of the area functional on a manifold.  This program has now been carried out to great success.  In \cite{IMN}, Irie, Marques, and Neves showed that Yau's conjecture is true for a generic metric on $M^{n+1}$ with $3 \le n+1\le 7$.  In fact, they proved more: generically the union of all minimal surfaces in $M$ is dense in $M$.  A crucial ingredient in the proof was the Weyl law for the volume spectrum proven by Liokumovich, Marques, and Neves \cite{LMN}. 

Later Marques, Neves, and Song \cite{MNS} improved the result in \cite{IMN} by showing that, for a generic metric on $M$, some sequence of minimal surfaces becomes equidistributed in $M$.   Gaspar and Guaraco \cite{GG} showed that the Weyl law and equidistribution results also hold in the Allen-Cahn setting.  In the non-generic case, Song \cite{Song} has shown that a closed manifold of dimension $3\le n+1\le 7$ with an arbitrary metric contains infinitely many minimal hypersurfaces.  Thus Yau's conjecture is fully resolved for these dimensions.  In the higher dimensional case, Li \cite{Li} has shown that for a generic metric on $M^{n+1}$ with $n+1\ge 8$ there are infinitely many minimal hypersurfaces of optimal regularity.   

Recently, Zhou \cite{Z} proved the multiplicity one conjecture of Marques and Neves \cite{MN}.  Using this, Marques and Neves \cite{MN} \cite{MN2} were able to prove the following: for a generic metric on $M^{n+1}$ with $3\le n+1\le 7$ there is a smooth, embedded, two-sided, index $p$ minimal hypersurface for every $p\in \N$.  Moreover, the area of these surfaces grows with $p$ according to the Weyl law for the volume spectrum \cite{LMN}.  

Min-max methods for constructing constant mean curvature  surfaces have only been developed more recently.  Fix a number $h > 0$.  Define a functional $A^h$ on open sets in $M$ with smooth boundary by setting
\[
A^h(\Omega) = \area(\bd \Omega) - h \vol( \Omega).  
\]
It is known that the critical points of $A^h$ are precisely those sets $\Omega$ whose boundary has constant mean curvature $h$ with respect to the inward pointing normal vector. In \cite{ZZ}, Zhou and Zhu developed a min-max theory for the $A^h$ functional, and used this theory to show that every closed manifold $M^{n+1}$ with $3\le n+1\le 7$ admits a smooth almost-embedded $h$-cmc hypersurface for every $h > 0$.   In \cite{ZZ1}, Zhou and Zhu extended the theory to construct more general prescribed mean curvature hypersurfaces.    Zhou \cite{Z} used this to give a proof of the multiplicity one conjecture of Marques and Neves \cite{MN}.  Earlier work of Chodosh and Mantoulidis \cite{CM} had shown that the multiplicity one conjecture was true for dimension $n+1=3$ in the Allen-Cahn setting.

Another technique for constructing minimal and constant mean curvature hypersurfaces is the so-called gluing method.  Starting from a collection of nearly minimal surfaces, one joins them together in a carefully chosen manner and then shows that the resulting surface can be perturbed to be minimal (or to have constant mean curvature).   Kapouleas and Yang \cite{KY} used this technique to construct minimal doublings of the Clifford torus in $S^3$.
Kapouleas has also used it to construct constant mean curvature surfaces of high genus in $\R^3$ \cite{K1}, and to construct minimal doublings of the equator in $S^3$ \cite{K}.
In \cite{PS}, Pacard and Sun used gluing methods to construct  constant mean curvature doublings of minimal hypersurfaces.   The following theorem is a special case of their results (see Theorem 2.1 and Corollary 2.1 in \cite{PS}).

\begin{theorem}[Pacard and Sun]
Let $n+1\ge 4$.  Let $\Sigma^n \subset M^{n+1}$ be an embedded minimal hypersurface.   Assume that the Jacobi operator $J$ for $\Sigma$ is invertible, and that the unique solution $\phi$ to $J\phi = 1$ does not change sign, and that $\phi$ has a non-degenerate critical point.  Then for every sufficiently small $\eps > 0$, there is an embedded $\eps$-cmc hypersurface which is a doubling of $\Sigma$. 
\end{theorem}

It is natural to ask whether surfaces produced by gluing methods can also be produced by variational techniques.  In the case of the Clifford torus in $S^3$, Ketover, Marques, and Neves \cite{KMN} proved the catenoid estimate and used it to give a min-max construction of the doublings of Kapouleas and Yang \cite{KY}. 
In this paper we show that, in certain circumstances, cmc doublings like those of Pacard and Sun can be constructed using min-max methods.  Our results apply in the case $3\le n+1\le 7$.  In the remainder of the introduction, we give a heuristic explanation of the min-max construction of cmc doublings.

\subsection{The Stable Case} Fix a dimension $3\le n+1\le 7$.  Suppose that $\Sigma^n\subset M^{n+1}$ is an embedded, two-sided, stable, minimal hypersurface.  Assume that a neighborhood of $\Sigma$ can be foliated by $\beta$-cmcs $\Sigma^\beta$ whose mean curvature vectors point towards $\Sigma$.  Every strictly stable minimal surface admits such a neighborhood by the implicit function theorem and the maximum principle.  A degenerate stable minimal surface may or may not admit such a neighborhood.

Let $\Omega^\eps$ be the open set in between $\Sigma^\eps$ and $\Sigma^{-\eps}$.  Then $\Omega^\eps$ is a critical point of $A^\eps$.  
Moreover, using the second variation formula for $A^\eps$, one can check that $\Omega^\eps$ is strictly stable for $A^\eps$.  
Thus $\Omega^\eps$ is a strict local minimum for $A^\eps$ in the smooth topology. 
Now, by the isoperimetric inequality, the empty set is also a local minimum for $A^\eps$.  Thus one can attempt to do min-max for the $A^\eps$ functional over all 1-parameter families of open sets connecting the empty set to $\Omega^\eps$. 

Theorem \ref{main1} and Theorem \ref{main2} are the main results of this paper in the stable case.  In Theorem \ref{main1}, we formalize the min-max argument outlined above to construct an $\eps$-cmc doubling of $\Sigma$.   
The key tool in the proof is the min-max theory for the $A^\eps$ functional introduced by Zhou and Zhu in \cite{ZZ}.  
We also borrow ideas from previous mountain pass type arguments for minimal surfaces.  See De Lellis and Ramic \cite{DR}, Marques and Neves \cite{MN1}, and Montezuma \cite{M}.  
In the case $n=2$, we are further able to show that the $\eps$-cmc doubling constructed in Theorem \ref{main1} consists of two parallel copies of $\Sigma$ joined by a small catenoidal neck.  This is the content of Theorem \ref{main2}.  The proof of this theorem is based on work of Chodosh, Ketover, and Maximo \cite{CKM}.  

\subsection{The Index 1 Case} Fix a dimension $3 \le n+1 \le 7$.  Let  $\Sigma^n \subset M^{n+1}$ be an embedded, two-sided, index 1, minimal hypersurface.  Let $L$ be the Jacobi operator on $\Sigma$ and assume that $L$ is non-degenerate and that the unique solution $\phi$ to $L\phi = 1$ is positive. 
The assumption that $L$ is non-degenerate together with the fact that $\phi > 0$ implies that a neighborhood of $\Sigma$ is foliated by cmc hypersurfaces.   Again let $\Sigma^\beta$ denote the $\beta$-cmc in this foliation and note that the mean curvature vector of $\Sigma^\beta$ points away from $\Sigma$.  Moreover, the surface $\Sigma^\beta$ lies at a height on the order of $\beta$ over $\Sigma$.  

Now fix a small number $\eps > 0$ and consider an $\eps$-cmc doubling $\Lambda^\eps$ of $\Sigma$.  If $\Lambda^\eps$ arises from the construction of Pacard and Sun, there is a decomposition
\[
\Lambda^\eps = \Lambda_+ \cup \Lambda_- \cup N
\]
where $N$ is a small neck, and $\Lambda_+$ and $\Lambda_-$ are each diffeomorphic to $\Sigma$ with a small ball removed.  The sheet $\Lambda_+$ is the graph of a function of small norm over $\Sigma^\eps$, and the sheet $\Lambda_-$ is the graph of a function of small norm over $\Sigma^{-\eps}$.   From this structure, we expect that the index of $\Lambda^\eps$ is three, where the three deformations decreasing $A^\eps$ correspond to varying the height of $\Lambda_+$, varying the height of $\Lambda_-$, and pinching the neck.  Thus $\Lambda^\eps$ should be the solution to a three parameter min-max problem.  

Based on this, we construct a three parameter family of surfaces $\Phi$ parameterized by the cube 
\[
X = \bigg\{(x,y,t):\, -\frac{\eps}2 \le x \le \frac \eps 2,\ -\frac \eps 2\le y \le \frac \eps 2,\ 0\le t\le R\bigg\},
\] 
where $R \gg \eps$ is a fixed small number. 
To define $\Phi$,  first let 
$
\Phi(0,0,0) =\Sigma^\eps \cup \Sigma^{-\eps}.
$
Think of this as a top sheet $\Sigma^\eps$ at height $\eps$ and a bottom sheet $\Sigma^{-\eps}$ at height $-\eps$.  Then extend $\Phi$ to the rest of $X$ as follows:  changing the $x$-coordinate varies the height of the top sheet by up to $\pm  \eps /2$, changing the $y$-coordinate varies the height of the bottom sheet by up to $\pm  \eps /2$, and increasing the $t$-coordinate opens up a neck between the two sheets.  

This family $\Phi$ has two important properties.  
\begin{itemize} 
\item[(i)] The surface $\Sigma^\eps \cup \Sigma^{-\eps}$ is an index two critical point of $A^\eps$ and the bottom face of the cube $X$ is a two parameter family of deformations that decreases $A^\eps$.\\  

\item[(ii)] The surface $\Sigma^\eps\cup \Sigma^{-\eps}$ maximizes $A^\eps$ over the boundary of $X$.  
\end{itemize}
To see property (ii), first observe that $\Sigma^\eps\cup \Sigma^{-\eps}$ maximizes $A^\eps$ over the bottom face of the cube.  Second, note that by opening a neck up to a fixed size $R \gg \eps$, we can ensure that $A^\eps(S) < A^\eps(\Sigma^\eps \cup \Sigma^{-\eps})$ for every surface $S$ in the top face of the cube.   Finally, consider a surface $T$ in the boundary of the bottom face of the cube.  Since $\Sigma$ is unstable, there is a uniform constant $c$  such that 
\[
A^{\eps}(T) < A^\eps(\Sigma^\eps \cup \Sigma^{-\eps}) - c\eps^2.
\]
On the other hand, by the catenoid estimate of Ketover, Marques, and Neves \cite{KMN}, it is possible to open a neck between the two sheets in $T$ without ever increasing the area by more than $C \eps^2/\vert \log \eps\vert$.  Therefore, we can ensure that $\Sigma^\eps \cup \Sigma^{-\eps}$ also maximizes $A^\eps$ over the side faces of the cube. 

Theorem \ref{main3} and Theorem \ref{main4} are the main results of this paper in the index 1 case.  In Theorem \ref{main3}, we construct $\eps$-cmc surfaces $\Lambda^\eps$ in $M$ by doing min-max for the $A^\eps$ functional over all families of surfaces $\Psi$ parameterized by the cube $X$ with $\Psi = \Phi$ on $\bd X$.  These surfaces $\Lambda^\eps$ have the property that $\area(\Lambda^\eps) \to 2\area(\Sigma)$ as $\eps \to 0$.  In Theorem \ref{main4}, we show that for a generic metric on $M$ the surfaces $\Lambda^\eps$ of Theorem \ref{main3} are doublings of $\Sigma$.

\subsection{Organization} The rest of the paper is organized as follows.  Section \ref{not} reviews some concepts from geometric measure theory as well as some definitions and theorems from Zhou's min-max theory.  Section \ref{sec3} constructs cmc doublings in the stable case.  Section \ref{sec4} constructs cmc doublings in the index 1 case.  Appendix \ref{QM} contains a quantitative minimality theorem that is needed to check that the width of certain homotopy classes is non-trivial.  Appendix \ref{gm} proves that a certain class of metrics is generic. 

\subsection{Acknowledgements} I would like to thank my advisor Andr\'e Neves for his continual encouragement and for many valuable discussions regarding this work.

\section{Notation and Preliminaries} 
\label{not} 

Let $M^{n+1}$ be a smooth, closed Riemannian manifold.  We begin by introducing some tools from geometric measure theory.   
\begin{itemize}
\item The set $\mathcal I_k(M,\Z_2)$ is the space of $k$-dimensional rectifiable flat chains mod 2 in $M$. \\
\item The flat norm on $\mathcal I_k(M,\Z_2)$ is denoted by $\mathcal F$, and the mass norm on $\mathcal I_k(M,\Z_2)$ is denoted by  $\mathbf M$. \\
\item Given $T\in \mathcal I_k(M,\Z_2)$, the notation $\vert T\vert$ stands for the varifold induced by $T$. \\
\item The $\mathbf F$ metric on $\mathcal I_{n+1}(M,\Z_2)$ is defined by 
\[
\mathbf F(\Omega_1,\Omega_2) = \fla(\Omega_1,\Omega_2) + \mathbf F(\vert \bd \Omega_1\vert, \vert \bd \Omega_2\vert)
\]
where $\mathbf F$ on the right hand side is Pitts' $\mathbf F$-metric on varifolds. \\
\item Following Marques and Neves, an embedded minimal cycle in $M$ is defined to be a varifold $V$ of the form 
\[
V = a_1\Gamma_1 + \hdots + a_\ell \Gamma_\ell
\]
where the $\Gamma_i$ are disjoint, smooth, embedded minimal surfaces in $M$ and the $a_i$ are positive integers. \\
\item Given $\eps > 0$, define $A^\eps\f \mathcal I_{n+1}(M,\Z_2) \to \R$ by 
$
A^\eps(\Omega) = \area(\bd \Omega) - \eps \vol(\Omega). 
$
\end{itemize} 

The following definitions are due to Zhou in \cite{Z}.  Let $X$ be a cubical complex and let $Z$ be a subcomplex of $X$.  Fix an $\mathbf F$-continuous map $\Phi\f X\to \mathcal I_{n+1}(M,\Z_2)$.

\begin{defn}
The $(X,Z)$-homotopy class of $\Phi$ consists of all sequences $\{\Psi_i\}_i$ with the following properties.  First, each $\Psi_i$ is an $\mathbf F$-continuous map $X\to \mathcal I_{n+1}(M,\Z_2)$. Second, for each $i$, there is a flat continuous homotopy $H_i\f [0,1]\times X\to \mathcal I_{n+1}(M,\Z_2)$ such that 
\begin{itemize}
\item[(i)] $H_i(0,x) = \Psi_i(x)$,\\
\item[(ii)] $H_i(1,x) = \Phi(x)$,\\
\item[(iii)] $\displaystyle \limsup_{i\to \infty} \left[\sup_{z\in Z,\ t\in[0,1]} \mathbf F(\Phi(z), H_i(t,z))\right] = 0$. 
\end{itemize} 
\end{defn} 

\begin{defn}
Let $\Pi$ be the $(X,Z)$-homotopy class of $\Phi$.  Fix an $\eps > 0$. Given a sequence $\{\Psi_i\}_i$ in $\Pi$ we let 
\[
L^\eps(\{\Psi_i\}_i) = \limsup_{i\to \infty} \left[\max_{x\in X} A^\eps(\Psi_i(x))\right].
\]
The width of the homotopy class $\Pi$ is then defined by
\[
L^\eps(\Pi) = \inf_{\{\Psi_i\}_i \in \Pi} L^\eps(\{\Psi_i\}_i).
\]
\end{defn}

\begin{defn}
\label{ae} 
Let $\Gamma$ be a smooth, immersed, constant mean curvature hypersurface in $M$.  Then $\Gamma$ is said to be almost-embedded provided for every point $p\in M$ either 
\begin{itemize}
\item[(i)] $\Gamma$ is embedded in a neighborhood of $p$, or\\
\item[(ii)] $\Gamma$ decomposes into the union of two embedded pieces $\Gamma_1$ and $\Gamma_2$ in a neighborhood of $p$ with $\Gamma_1$ on one side of $\Gamma_2$. 
\end{itemize}
\end{defn}

The following min-max theorem for the $A^\eps$ functional is due to Zhou.  See Theorem 1.7 and Theorem 3.1 in \cite{Z}. 
\begin{theorem}[Zhou]
\label{mm} 
Assume that the min-max width $\Pi$ is non-trivial, i.e., that 
\[
L^\eps(\Pi) > \max_{z\in Z} A^\eps(\Phi(z)).
\]
Then there is a smooth, almost-embedded $\eps$-cmc hypersurface $\Lambda^\eps$ in $M$, and there is an open set $\Theta^\eps$ in $M$ with $\bd \Theta^\eps = \Lambda^\eps$ and $A^\eps(\Theta^\eps) = L^\eps(\Pi)$.  Moreover, the index of $\Lambda^\eps$ as a critical point of $A^\eps$ is at most the dimension of $X$. 
\end{theorem}

\section{The Stable Case}
\label{sec3}

\subsection{Statement of Results} 

We now formalize the assumptions outlined in the introduction.  Fix a dimension $3 \le n+1\le 7$.  Let $(M^{n+1},g)$ be a closed Riemannian manifold and let $\Sigma^n \subset M^{n+1}$ be a closed, connected, two-sided, minimal hypersurface.  Also assume the following. 

\begin{itemize}
\item[(S-i)] There is a neighborhood $U$ of $\Sigma$ and a smooth function $f$ on $U$ and a number $\alpha > 0$ such that $-\alpha < f < \alpha$ on $U$.\\
\item[(S-ii)]  The level set $\Sigma^\beta := f\inv(\beta)$ is a closed hypersurface diffeomorphic to $\Sigma$ with constant mean curvature $\vert \beta\vert$ for $\vert \beta \vert < \alpha$. Moreover $\Sigma^0 = \Sigma$. \\
\item[(S-iii)] The mean curvature vector of $\Sigma^\beta$ points toward $\Sigma$ for each $\vert \beta\vert < \alpha$. \\
\item[(S-iv)] The gradient $\grad f$ does not vanish anywhere on $U\setminus \Sigma$. 
\end{itemize}
For future reference, we will refer to this collection of assumptions as (S).  Let $\Omega^\eps$ be the region contained between $\Sigma^\eps$ and $\Sigma^{-\eps}$. 

Our main theorems in the stable case are the following. 

\begin{theorem}
\label{main1}
Fix $(M^{n+1},g)$ and $\Sigma$ for which the assumptions (S) hold.  Then there is a smooth, almost-embedded $\eps$-cmc $\Lambda^\eps$ contained in $\Omega^\eps$.  Moreover, there is an open set $\Theta^\eps \subset \Omega^\eps$ with $\Lambda^\eps = \bd \Theta^\eps$, and  
the index of $\Lambda^\eps = \bd \Theta^\eps$ as a critical point of $A^\eps$ is at most 1. 
\end{theorem}

\begin{theorem}
\label{main2}
Assume further that $n=2$.  Then the surface $\Lambda^\eps$ from the previous theorem admits a decomposition 
\[
\Lambda^\eps = \Lambda^{\eps}_+ \cup \Lambda^{\eps}_- \cup N
\]
where each $\Lambda^\eps_{\pm}$ is the graph of a function of small norm over $\Sigma$ minus a ball and $N$ is a catenoidal neck. 
\end{theorem}

\subsection{Sweepouts} 
We would like to use a mountain pass type argument to produce an $\eps$-cmc.   
We now introduce the maps that will serve as sweepouts.  Fix a number $0 < \eps < \alpha$. For each $0 < \beta < \alpha$, let $\Omega^\beta = \{-\beta < f < \beta\}$ denote the open set between $\Sigma^{-\beta}$ and $\Sigma^{\beta}$.  Also fix a small number $\eta > 0$ to be specified later and let $\Omega^* = \Omega^{\eps+\eta}$. 

\begin{prop}
\label{map}
There is an $\mathbf F$-continuous map $\Phi\f [0,1] \to \mathcal I_{n+1}(\Omega^*,\Z_2)$ with $\Phi(0) = \emptyset$ and $\Phi(1) = \Omega^\eps$. 
\end{prop}

\begin{proof}
The map $\Psi\f [0,1] \to \mathcal I_{n+1}(\Omega^*,\Z_2)$ given by $\Psi(t) = \Omega^{t\eps}$ is continuous in the flat topology.   By Lemma A.1 in Zhou and Zhu \cite{ZZ}, it is possible to construct a sequence $\{\phi_i\}_i$ of better and better discrete approximations to $\Psi$.  Applying Zhou's discrete to continuous interpolation theorem (Theorem 1.12 in \cite{Z}) produces the required map $\Phi$ from the  sequence $\{\phi_i\}_i$. 
\end{proof}

\begin{defn}
\label{h}
Let $\Phi$ be the map constructed in the previous proposition.  
A sweepout is an $\mathbf F$-continuous map $\Psi\f [0,1]\to \mathcal I_{n+1}(\Omega^*,\Z_2)$ with $\Psi(0) = \emptyset$ and $\Psi(1) = \Omega^\eps$ that is flat homotopic to $\Phi$ relative to $\bd[0,1]$.  More precisely, this last statement means that there is a flat continuous map $H\f [0,1]\times [0,1] \to \mathcal I_{n+1}(\Omega^*,\Z_2)$ such that
\begin{itemize}
\item[(i)] $H(0,t) = \Phi(t)$,
\item[(ii)] $H(1,t) = \Psi(t)$,
\item[(iii)] $H(s,0) = \emptyset$,
\item[(iv)] $H(s,1) = \Omega^\eps$, 
\end{itemize} 
for all $s$ and $t$. 
\end{defn}

\begin{rem}
Let $X = [0,1]$ and $Z = \{0,1\}$.  Note that a sweepout $\Psi$ is essentially an element of the $(X,Z)$-homotopy class of $\Phi$ as defined in Section \ref{not}.  However, we require that $\Psi(0)$ exactly equals $\Phi(0)$ and that $\Psi(1)$ exactly equals $\Phi(1)$.  Moreover, all sets in a sweepout $\Psi$ are required to be contained in the set $\Omega^*$. 
\end{rem}

\begin{defn}
The min-max width $W^\eps$ is defined by 
\[
W^\eps = \inf_{\text{sweepouts } \Psi} \bigg[ \sup_{t\in[0,1]} A^\eps(\Psi(t))\bigg].
\]
\end{defn}

\begin{defn}
A critical sequence is a sequence of sweepouts $\{\Psi_i\}_i$ with the property that 
\[
\lim_{i\to \infty} \bigg[\sup_{t\in[0,1]} A^\eps(\Psi_i(t)) \bigg] = W^\eps. 
\]
\end{defn} 

\begin{defn}
Let $\{\Psi_i\}_i$ be a critical sequence.  The associated critical set $C(\{\Psi_i\}_i)$ is the collection of all varifolds of the form 
\[
V = \lim_{i\to \infty} \vert \bd \Psi_i(t_i) \vert 
\]
with $t_i \in [0,1]$ and $\lim_{i\to \infty} A^\eps(\Psi_i(t_i)) = W^\eps$.  Note that the critical set is always non-empty and compact. 
\end{defn}

\subsection{Non-trivial Width} 

Fix $(M,g)$ and $\Sigma$ satisfying the assumptions (S) and fix a number $0 <\eps < \alpha$.  Recall that the notation $\Omega^\beta$ denotes the open set between $\Sigma^{-\beta}$ and $\Sigma^\beta$.  Also $\eta > 0$ is a fixed small number and $\Omega^* = \Omega^{\eps + \eta}$.  The number $W^\eps$ is the min-max width of the collection of all paths in $\mathcal I_{n+1}(M,\Z_2)$ joining $\emptyset$ to $\Omega^\eps$ while staying inside $\Omega^*$.  

The goal of this section is to show that $W^\eps > \max\{A^\eps(\Omega^\eps),0\}$.  The fact that $W^\eps > A^\eps(\emptyset) = 0$ is a consequence of a suitable isoperimetric inequality. 

\begin{prop}[See Theorem 2.15 in \cite{ZZ}]
There are constants $C$ and $V$ such that 
\[
\area(\bd \Omega) \ge C \vol(\Omega)^{n/(n+1)}
\]
whenever $\Omega \in \mathcal I_{n+1}(M,\Z_2)$ satisfies $\vol(\Omega) < V$. 
\end{prop}

\begin{corollary}
\label{p}
The width $W^\eps$ is positive. 
\end{corollary}

\begin{proof}
Choose a small number $0 < v < \min\{V,\vol(\Omega^\eps)\}$. 
Let $\Psi\f [0,1] \to \mathbf{I}_{n+1}(\Omega^*,\Z_2)$ be a sweepout.  By continuity, there must be some $\Omega$ in the image of $\Psi$ with 
$
\vol(\Omega) = v.
$
It follows that 
\begin{align*}
A^\eps(\Omega) &= \area(\bd \Omega) - \eps\vol(\Omega)\\ 
& \ge C\vol(\Omega)^{n/(n+1)} - \eps \vol(\Omega) = \vol(\Omega)^{n/(n+1)} \left(C - \eps \vol(\Omega)^{1/(n+1)}\right).
\end{align*} 
The number on the right hand side is positive provided $v$ is taken sufficiently small.
\end{proof}

It remains to show that $W^\eps > A^\eps(\Omega^\eps)$.  To begin, we first show that $\Omega^\eps$ is a strictly stable critical point of $A^\eps$.  

\begin{prop}
\label{ss}
Assume that $(M,g)$ and $\Sigma$ satisfy the assumptions (S).  
Then $\Omega^\eps$ is a strictly stable critical point of $A^\eps$. 
\end{prop} 

\begin{proof}
Let $N$ be the outward pointing normal vector to $\bd \Omega^\eps$.  The second variation formula for $A^\eps$ says that 
\[
\delta^2 A^\eps\eval_{\Omega^\eps}(uN) = -\int_{\Sigma^\eps} u L_\eps u - \int_{\Sigma^{-\eps}} uL_{-\eps} u ,
\]  
where $L_\beta$ is the Jacobi operator on $\Sigma^\beta$.   Hence to prove the claim it suffices to show that the lowest eigenvalue of $L_\beta$ is positive for $\beta = \pm \eps$. 

We will prove this for $L_\eps$, the argument for $L_{-\eps}$ being essentially identical.   Let $H$ be the mean curvature operator on $\Sigma^\eps$ (computed with respect to $N$).  It is known that $L_\eps$ is the linearization of $H$.  
For $\gamma$ close enough to $\eps$, we can write $\Sigma^\gamma$ as a normal graph of a function $\varphi_\gamma$ over $\Sigma^\eps$.  
Define 
\[
\psi = \frac{d}{d\gamma}\eval_{\gamma = \eps}(\varphi_\gamma)
\]
and note that $\psi \ge 0$.  
Differentiating the equation $H(\varphi_\gamma) = -\gamma$ and evaluating at $\gamma = \eps$ shows that 
$
L_\eps\psi = -1. 
$

The existence of a non-negative solution to this equation implies that the lowest eigenvalue of $L_\eps$ is positive.  Indeed, let $\lambda$ be the lowest eigenvalue of $L_\eps$ and let $\zeta > 0$ be the associated eigenfunction so that $L_\eps\zeta + \lambda \zeta = 0$.  Since 
\[
\int_{\Sigma^\eps} \zeta = -\int_{\Sigma^\eps} \zeta L_\eps \psi = -\int_{\Sigma^\eps} \psi L_\eps \zeta = \lambda \int_{\Sigma^\eps} \psi \zeta
\]
it follows that $\lambda$ must be positive.
\end{proof} 

The desired inequality for the width now follows from the quantitative minimality results in Appendix \ref{QM}.

\begin{prop}
\label{qms}
There are positive constants $\delta$ and $C$ such that 
\[
A^\eps(\Omega) \ge A^\eps(\Omega^\eps) + C \fla(\Omega, \Omega^\eps)^2
\]  
for all $\Omega \in \mathcal I_{n+1}(\Omega^*,\Z_2)$ with $\fla(\Omega,\Omega^\eps) < \delta$. 
\end{prop} 

\begin{proof}
Proposition \ref{ss} says that $\Omega^\eps$ is strictly stable for $A^\eps$. Hence the desired result follows from Corollary \ref{qmss} in Appendix \ref{QM}. 
\end{proof}

\begin{corollary} 
\label{p1}
The width satisfies $W^\eps > A^\eps(\Omega^\eps)$.
\end{corollary}

\begin{proof}
Let $\delta$ and $C$ be the constants from Proposition \ref{qms}. 
Without loss we can assume that $\delta < \vol(\Omega^\eps)$. 
Let \[
\Psi\f [0,1] \to \mathcal I_{n+1}(\Omega^*,\Z_2)
\]
be a sweepout.   By continuity there is some $\Omega$ in the image of $\Psi$ with $\fla(\Omega, \Omega^\eps) = \delta/2$.  But then Proposition \ref{qms} implies that
\[
A^\eps(\Omega) \ge A^\eps(\Omega^\eps) + \frac{C\delta^2}{4},
\]
and the corollary follows.
\end{proof} 

\subsection{A Deformation Lemma}

The goal of this section is to prove a deformation lemma that will be used to show that the min-max surface lies in the interior of $\Omega^*$.  The proof closely follows an argument of Marques and Neves \cite{MN}, 
and relies on the existence of a deformation that pushes currents away from $\bd \Omega^*$ while simultaneously decreasing $A^\eps$.   

\begin{prop}
It is possible to find an open set $\Omega^{**}$ with  
\[
\Omega^\eps \cc \Omega^{**} \cc \Omega^*
\]
together with a Lipschitz vector field $Z$ supported on $\Omega^* \setminus \Omega^\eps$ with flow $\varphi_t$ such that the following properties hold.  
\begin{itemize}
\item[(i)] $\supp((\varphi_1)_\# \Omega) \subset \Omega^{**},\;\, \text{for all } \Omega\in \mathcal I_{n+1}(\Omega^*,\Z_2)$
\item[(ii)] $A^\eps((\varphi_1)_\# \Omega) \le A^\eps(\Omega),\; \text{for all } \Omega\in \mathcal I_{n+1}(\Omega^*,\Z_2) $
\end{itemize}  
\end{prop} 

\begin{proof}
Recall that the cmc foliation near $\Sigma$ is given by the level sets of a function $f$ and that $\grad f \neq 0$ on a neighborhood $W$ of $\Sigma^{-\eps}\cup \Sigma^{\eps}$.   By taking $\eta$ small enough, we can assume that $\Omega^*\setminus \Omega^\eps \subset W$.  Define a vector field $X = \grad f / \vert \grad f\vert^2$ on $W$.  Then define 
\[
Z = \begin{cases}
-(f-\eps)X, &\text{on } \Omega^*\setminus \Omega^\eps\\
0, &\text{otherwise}
\end{cases}
\]
and note that $Z$ is a Lipschitz vector field on $\Omega^*$. 

Fix some $\Omega \in \mathcal I_{n+1}(\Omega^*,\Z_2)$ and let $\nu$ be the outward pointing normal vector to $\bd \Omega$.  According to the first variation formula, 
\[
\delta A^\eps\eval_{\Omega}(Z) = \int_{\bd \Omega} \div_{\sigma}Z  - \eps \int_{\bd \Omega} \la Z, \nu\ra  \, d\mathcal H^n.
\]  
To understand the right hand side, we need to compute $\div_\sigma Z$.  

Let $\phi_t$ denote the flow of $X$ and let $x$ denote a point in $\Sigma^\eps$.  Choose a point $y = \psi(x,t)$ and let $\sigma\subset T_y M$ be an $n$-plane.  Let $x_i$ be coordinates on a neighborhood of $x$ in $\Sigma^\eps$.  Then the map $\psi(x,t) = \phi_t(x)$ gives coordinates on a neighborhood of $y$.  Define $e_i = \bd \psi/\bd x_i$ and note that $\bd \psi/\bd t= X$.  Let $N = \grad f /\vert \grad f \vert$ be the unit normal vector to the surfaces $\Sigma^\beta$ and let $A$ denote the second fundament form of the surfaces $\Sigma^\beta$.  As in Marques and Neves \cite{MN}, we compute 
\begin{gather*}
\la \del_{e_i} Z, e_j\ra = -\la Z, A(e_i,e_j)\ra, \\
\la \del_{N}Z, N\ra = \la \del_N (-(f-\eps)X), N\ra = -1 - (f-\eps)\la \del_N X,N\ra. 
\end{gather*}
Also we have 
\begin{gather*}
 \la \del_{e_i}Z, N\ra  = 
(f-\eps) \left\la \frac{\bd \psi}{\bd x_i} , \del_{\frac{\bd \psi}{\bd t}} N\right\ra = \frac{f-\eps}{\vert \grad f\vert} \la e_i, \del_N N\ra,\\
\la e_i, -\del_N Z\ra = \left\la e_i, N\left(\frac{f-\eps}{\vert \grad f\vert}\right)N + \frac{f-\eps}{\vert \grad f\vert} \del_N N\right\ra = \frac{f-\eps}{\vert \grad f\vert }\la e_i, \del_N N\ra,
\end{gather*}
and so 
\[
\la \del_{e_i} Z, N\ra = -\la e_i, \del_N Z\ra.
\] 
Using this one can compute $\div_\sigma Z$ as follows.

Let $v_1,\hdots,v_n$ be an orthonormal basis for $\sigma$.  We can arrange that $v_1,\hdots,v_{n-1}$ are tangent to $\Sigma^{\eps+t}$ and that $v_n = (\cos \theta)u + (\sin \theta)N$ for some unit vector $u$ which is tangent to $\Sigma^{\eps+t}$ and orthogonal to $v_1,\hdots,v_{n-1}$.  Let $H$ be the mean curvature vector for $\Sigma^{\eps+t}$.  Then from the above computations one finds 
\begin{align*}
\div_\sigma Z &= \left(\la \del_u Z, u\ra +\sum_{i=1}^{n-1} \la \del_{v_i} Z, v_i\ra\right) + \la \del_{v_n} Z, v_n\ra - \la \del_u Z, u\ra \\
&= -\la Z, H\ra + (\cos^2 \theta - 1)\la \del_u Z,u\ra + \sin^2\theta \la \del_N Z, N\ra \phantom{\bigg(}\\
&= -\frac{\eps(f-\eps)}{\vert \grad f\vert} - \sin^2\theta \bigg(1 + (f-\eps) \la \del_N X, N\ra + (f-\eps) \la X , A(u,u)\ra \bigg). 
\end{align*} 
Therefore, provided $\eta$ is small enough, it follows that 
\begin{align*}
\div_\sigma Z - \eps \la Z,\nu\ra \le -\frac{\eps(f-\eps)}{\vert \grad f\vert} + \eps\vert Z\vert = 0. 
\end{align*}
Hence following the flow of $Z$ decreases $A^\eps$.  
\end{proof}

\begin{corollary}
\label{gpt}
There exists an open set $\Omega^{**}$ with $\Omega^\eps \cc \Omega^{**} \cc \Omega^*$ and a critical sequence $\{\Psi_i\}_i$ such that 
\[
\supp(\Psi_i(x)) \subset \Omega^{**}
\]
for all $i$ and all $x\in [0,1]$. 
\end{corollary}

\begin{proof}
Let $\{\Phi_i\}_i$ be a criticial sequence.  Let $\varphi_t$ denote the flow of $Z$.  Define $\Psi_i(x) = (\varphi_1)_\# \Phi_i(x)$ for $x\in [0,1]$.  By the previous proposition, $\{\Psi_i\}_i$ is as required. 
\end{proof}

\subsection{Constructing the Min-Max Surfaces}

We can now perform a min-max argument to construct the  doublings.  The following min-max theorem essentially follows from Theorem \ref{mm}.  The theorem is not an immediate consequence of Theorem \ref{mm} because we require that the surfaces in a sweepout are contained in $\Omega^*$.  However, it is straightforward to modify the proof of Theorem \ref{mm} to handle our situation.

\begin{theorem}
\label{mms}
Assume that $W^\eps > \max\{0, A^\eps(\Omega^\eps)\}$.  Then for any critical sequence $\{\Psi_i\}_i$ there is a varifold $V\in C(\{\Psi_i\}_i)$ that is induced by a smooth, almost-embedded $\eps$-cmc hypersurface $\Lambda^\eps$.  There is an open set $\Theta^\eps\subset \Omega^*$ such that $\bd \Theta^\eps = \Lambda^\eps$ and $A^\eps(\Theta^\eps) = W^\eps$.  Moreover, there is a bound $\ind(\Lambda^\eps)\le 1$. 
\end{theorem}

\begin{proof}
We outline the necessary changes to the proof of Theorems 1.7 and 3.1 in \cite{Z}.  Let $X = [0,1]$ and $Z = \{0,1\}$.  Let $\Phi$ be the map from Proposition \ref{map}.  Zhou defines the $(X,Z)$-homotopy class of $\Phi$ to consist of all sequences $\{\Psi_i\}_i$ such that each $\Psi_i$ is flatly homotopic to $\Phi$ and 
\[
\lim_{i\to \infty} \max\{\mathbf F(\Psi_i(0),\emptyset), \mathbf F(\Psi_i(1),\Omega^\eps)\} = 0. 
\]
However, because the domain $X$ is one dimensional, the interpolation results of Zhou show that nothing changes if we instead insist that $\Psi_i(0) = \emptyset$ and $\Psi_i(1) = \Omega^\eps$ for all $i$.  This leads to the notion of homotopy in Definition \ref{h}. 

Now let $\Psi$ be a sweepout.  Assume that $\Psi'$ is obtained from $\Psi$ by either the pulltight procedure, the combinatorial argument, or the deformations in the index estimates.  Note that we can arrange so that the following property is true: if $W$ is an open set and $\supp(\Psi(t)) \subset W$ for all $t\in [0,1]$ then $\supp(\Psi'(t)) \subset W'$ for all $t\in [0,1]$ where $W'$ is a slightly larger open set containing $W$.  Therefore, by Corollary \ref{gpt}, we can perform all the arguments of Zhou on a critical sequence $\{\Psi_i\}_i$ while always staying inside $\Omega^*$.
\end{proof}

We can now prove the first main theorem.  

\begin{proof} (Theorem \ref{main1}) 
Corollary \ref{p} and Corollary \ref{p1} show that 
\[
W^\eps > \max\{0,A^\eps(\Omega^\eps)\}.
\]
Therefore Theorem \ref{mms} applies to produce $\Lambda^\eps$ and $\Theta^\eps$ satisfying the conclusion of Theorem \ref{main1}. 
\end{proof}

\subsection{Topology of the Min-Max Doubling} 

The goal of this section is to show that the min-max surfaces constructed above consist of two parallel copies of $\Sigma$ joined by a small catenoidal neck.  
For this section only, we require that $n+1=3$.  

Choose a sequence $\eps_j \to 0$.  Let $\Lambda_j = \Lambda^{\eps_j}$ be the $\eps_j$-cmc given by Theorem \ref{main1}.  
Note that $\Lambda_j$ converges to $\Sigma$ in the Hausdorff distance.  Hence by the compactness theorem for cmcs with bounded area and index (Zhou \cite{Z}),
 there is a point $p\in \Sigma$ such that (up to a subsequence) $\Lambda_j$ converges locally smoothly to $\Sigma$ away from $p$. 

\begin{prop}
The convergence $\Lambda_j \to \Sigma$ occurs with multiplicity two.  
\end{prop}

\begin{proof}
First we show that the multiplicity is at most two.  To prove this, it suffices to show that 
\[
\limsup_{\eps \to 0} W^\eps \le 2\area(\Sigma). 
\]
Fix some $\eps >0$.  Since the map $\Phi:[0,1]\to \mathcal I_{n+1}(M,\Z_2)$ given by $\Phi(t) = \Omega^{t\eps}$ can be interpolated to a sweepout, it follows that 
\[
W^\eps \le \max_{\beta \in [0,\eps]} A^\eps(\Omega^\beta) \le \max_{\beta\in [0,\eps]} \area(\bd \Omega^\beta).
\]
The quantity on the right hand side converges to $2\area(\Sigma)$ as $\eps \to 0$. 

It remains to show that the multiplicity is at least 2.  To prove this, it suffices to show that 
\[
\liminf_{\eps \to 0} W^\eps \ge 2\area(\Sigma).
\]
To see this, recall that 
\[
W^\eps \ge A^\eps(\Omega^\eps) = \area(\bd \Omega^\eps) - \eps \vol(\Omega^\eps). 
\]
Again the quantity on the right hand side converges to $2\area(\Sigma)$ as $\eps \to 0$. 
\end{proof}

\begin{prop}
The surface $\Lambda_j$ is connected.
\end{prop}

\begin{proof}
Otherwise there would be a component $\Lambda_j'$ of $\Lambda_j$ which is graphical over $\Sigma$.  The maximum principle shows that such a surface $\Lambda_j'$ cannot exist. 
\end{proof}

\begin{corollary}
The index of $\Lambda_j$ is one.
\end{corollary}

\begin{proof}
Suppose to the contrary that $\ind(\Lambda_j) = 0$.  By the curvature estimates for stable cmcs (see Zhou \cite{Z}), the convergence $\Lambda_j \to \Sigma$ would consequently occur smoothly everywhere.  But, since $\Sigma$ is two-sided, it is impossible for a connected surface $\Lambda_j$ to converge smoothly to $\Sigma$ with multiplicity two.
\end{proof}

We can now give the proof of Theorem \ref{main2}. 

\begin{proof} (Theorem \ref{main2}) 
The proof is based on results of Chodosh, Ketover, and Maximo \cite{CKM}.  Although the results in \cite{CKM} are stated for minimal hypersurfaces, one can check that they continue to hold in our setting.  For the sake of completeness, we sketch the details of the argument. 

Let $A_j$ denote the second fundamental form of $\Lambda_j$. Recall that stable cmcs have curvature estimates (see Zhou \cite{Z}).  Therefore we must have 
\[
\lim_{j\to\infty} \max_{x\in \Lambda_j} \vert A_j(x)\vert = \infty
\]
since the convergence $\Lambda_j \to \Sigma$ is not smooth near $p$.  By a point picking argument together with the fact that $\ind(\Lambda_j) = 1$, it is possible to find a constant $C > 0$ and a sequence of points $p_j \in \Lambda_j$ with $\vert A_j(p_j)\vert \to \infty$ and such that 
\begin{equation*}
\label{cb}
\vert A_j(x)\vert \dist_{M}(x,p_j) \le C 
\end{equation*}
for all $x\in \Lambda_j$.  Moreover, it is clear that $p_j \to p$.

Fix a small number $\sigma > 0$.
Choose a sequence $\eta_j\to 0$ for which $\dist_M(p_j, p) <  \eta_j$ and 
\[
\lim_{j\to \infty} \eta_j \vert A_j(p_j)\vert = \infty. 
\] 
We claim that for $j$ sufficiently large there is a bound 
\[
\vert A_j(x)\vert \dist_M(x,p_j)  \le \frac{1}{4} 
\]
for all $x\in \Lambda_j \cap (B(p,\sigma) \setminus B(p_j,\eta_j))$.  Suppose not.  Then there would be points $x_j \in \Lambda_j \cap (B(p,\sigma) \setminus B(p_j,\eta_j))$ with 
\[
\vert A_j(x_j)\vert  \dist_M(x_j,p_j)> \frac 1 4.
\]
Let $\Lambda_j'$ be the surface $\Lambda_j$ rescaled by a factor $\dist_M(x_j,p_j)\inv$ about the point $p_j$.  Let $ A_j'$ denote the 2nd fundamental form of $\Lambda_j'$, and given a point $x\in \Lambda_j$ let $x'$ denote the corresponding point in $\Lambda_j'$. 

Notice that 
\[
\vert A_j'(x')\vert = {\vert A_j(x)\vert} {\dist_M(x_j,p_j)},  
\]
and hence the surfaces $A_j'$ have uniform curvature bounds on compact sets that do not include the origin.  Moreover, 
\[
\vert A_j'(0)\vert \ge  {\vert A_j(p_j)\vert} \eta_j \to \infty
\]
as $j\to \infty$. Therefore, (up to a subsequence) the surfaces $\Lambda_j'$ converge locally smoothly away from the origin to a complete, embedded minimal surface $\Lambda'$ with multiplicity two.  Since the mean curvature vectors of the two sheets of $\Lambda_j'$ point toward each other, it follows that $\Lambda'$ must be stable.  Hence $\Lambda'$ is a plane.  But this means that $\vert A_j'(x_j')\vert \to 0$, and this contradicts the way the points $x_j$ were chosen.

Next one combines the preceding curvature estimate with a Morse theory argument (Lemma 3.1 in \cite{CKM}) to conclude that $\Lambda_j \cap B(p,\sigma)$ and $\Lambda_j \cap B(p_j,\eta_j)$ have the same topology.  We are now reduced to showing that $\Lambda_j \cap B(p_j,\eta_j)$ is topologically a catenoid.    Let $\Lambda_j''$ be the surface $\Lambda_j$ rescaled by a factor $\eta_j\inv$ about the point $p_j$.  It is equivalent to check that $\Lambda_j''\cap B(0,1)$ is a catenoid. 

Let $\Lambda_j'''$ be the surface $\Lambda_j''$ rescaled by a factor $\vert A_j''(0)\vert$ about the origin.  Then $\Lambda_j'''$ has uniform curvature estimates everywhere.  Thus (up to a subsequence) the surfaces $\Lambda_j'''$ converge locally smoothly to a complete, embedded, two-sided, non-flat minimal hypersurface $\Lambda'''\subset \R^3$.  Moreover, we have $\ind(\Lambda''') \le 1$.  By the results in \cite{FCS} and  \cite{LR}, it follows that $\Lambda'''$ is a catenoid.  Fix a radius $R > 0$ so that $\vert A'''(y)\vert \dist(y,0) < 1/4$ for all $y \in \Lambda''' \setminus B(0,R)$.

We claim that for $j$ sufficiently large there is a bound 
\[
\vert A_j''(y)\vert \dist(y,0) \le \frac{1}{4}
\]
for all $y\in \Lambda_j'' \cap (B(0,2) \setminus B(0,R/\vert A_j''(0)\vert))$.  Suppose not.  Then there would be points $y_j \in \Lambda_j'' \cap (B(0,2)\setminus B(0,R/\vert A_j''(0)\vert))$ with 
\[
\vert A_j''(y_j)\vert \dist(y_j, 0) > \frac{1}{4}.
\]
Let $\Lambda_j''''$ be the surface obtained by scaling $\Lambda_j''$ by a factor $\dist(y_j,0)\inv$ about the origin.  

We claim that $\vert A_j''''(0)\vert \to \infty$ as $j\to \infty$.  Suppose this were not the case.  Then since 
\[
\vert A_j''''(0)\vert = \vert A_j''(0)\vert \dist(y_j,0),
\]
it must be that 
\[
\frac{R}{\vert A_j''(0)\vert} \le \dist(y_j,0) \le \frac{B}{\vert A_j''(0)\vert}
\]
for some constant $B$.  But then (up to a subsequence) $\Lambda_j''''$ must converge to a surface $\Lambda'''' = a\Lambda'''$ where 
\[
\frac{1}{B} \le a \le \frac{1}{R}.
\]   
Now observe that 
\[
\frac 1 4 < \vert A''''(\dist(y_j,0)\inv y_j)\vert = a\inv \vert A'''(a\inv \dist(y_j,0)\inv y_j)\vert. 
\]
This contradicts the choice of $R$. Therefore it must be that $\vert A_j''''(0)\vert \to \infty$ as $j\to \infty$. 

The surfaces $\Lambda_j''''$ have uniform curvature estimates on compact subsets that do not include the origin.  Hence arguing as above, it follows that (up to a subsequence) the surfaces $\Lambda_j''''$ converge locally smoothly to a plane away from the origin.  This contradicts the way the points $y_j$ were chosen.  Finally one repeats the Morse theory argument with this curvature estimate to deduce that $\Lambda_j'' \cap B(0,1)$ has the same topology as $\Lambda_j'' \cap B(0,R/\vert A_j''(0)\vert)$.  Since the surface $\Lambda_j''\cap B(0,R/\vert A_j''(0)\vert)$ has the same topology as $\Lambda_j''' \cap B(0,R)$, it follows that $\Lambda_j''\cap B(0,R/\vert A_j''(0)\vert)$ is topologically a catenoid, as needed.  This completes the proof of Theorem \ref{main2}. \end{proof} 

\section{The Index 1 Case} 
\label{sec4}

\subsection{Statement of Results} Now consider the index 1 case.  Fix a dimension $3 \le n+1\le 7$.  Let $(M^{n+1},g)$ be a closed Riemannian manifold and let $\Sigma^n \subset M^{n+1}$ be a closed, connected, two-sided, minimal hypersurface. 
Also assume the following. 
\begin{itemize}
\item[(U-i)]  The hypersurface $\Sigma$ has index 1 and the Jacobi operator $L$ for $\Sigma$ is non-degenerate.  Moreover, the unique solution $\phi$ to $L\phi = 1$ is positive. 
\end{itemize}
Note that by assumption (U-i) and the implicit function theorem, there is a neighborhood of $\Sigma$ that is foliated by constant mean curvature hypersurfaces whose mean curvature vectors point away from $\Sigma$. More precisely, we have the following.
\begin{itemize}
\item[(i)] There is a neighborhood $U$ of $\Sigma$ and a smooth function $f\f U\to (-\beta,\beta)$.\\
\item[(ii)] For each $\eps\in (-\beta,\beta)$, the set $\Sigma^\eps = f\inv(\eps)$ is a smooth hypersurface with constant mean curvature $\vert \eps\vert$.  Moreover, $\Sigma^0 = \Sigma$. \\

\item[(iii)] For each $\eps\in(-\beta,\beta)$, the mean curvature vector of $\Sigma^\eps$ points away from $\Sigma$. 
\end{itemize}

The next theorem is the main result of the paper in the index 1 case.  

\begin{theorem}
\label{main3}
Fix $(M,g)$ and $\Sigma$ for which the assumption (U-i) holds.  Then for each small $\eps >0$, there is a smooth, almost-embedded hypersurface $\Lambda^\eps$ of constant mean curvature $\eps$ in $M$.  The index of $\Lambda^\eps$ is at most 3 and $\area(\Lambda^\eps) \to 2\area(\Sigma)$ as $\eps \to 0$. 
\end{theorem}

To ensure that $\Lambda^\eps$ is a doubling of $\Sigma$, we have to make an additional assumption.  Namely, suppose the following additional property holds. 
\begin{itemize}
\item[(U-ii)] The varifold $2\Sigma$ is the only embedded minimal cycle in $M$ with area $2\area(\Sigma)$.
\end{itemize} 
Then we have the following. 

\begin{theorem}
\label{main4} 
Fix $(M,g)$ and $\Sigma$ for which the assumptions (U-i) and (U-ii) hold.  Then the surfaces $\Lambda^\eps$ from Theorem \ref{main3} converge to $2\Sigma$ as varifolds as $\eps \to 0$. 
\end{theorem}

\begin{rem}
It is natural to ask whether hypothesis (U-ii) significantly restricts the applicability of Theorem \ref{main4}.  In Appendix \ref{gm} we show that (U-ii) holds for a generic set of metrics on $M$. 
\end{rem} 

\subsection{Construction of the three parameter family} 
In this section, we formally construct the three parameter family $\Phi$ described in the introduction.  Fix $(M,g)$ and $\Sigma$ satisfying the assumption (U-i) and fix a small number $\eps > 0$.  For simplicity, we give the construction in the case where $n+1 =3$.  The cases $4 \le n+1 \le 7$ are similar but easier since one can use cylindrical necks rather than catenoidal ones.

Before constructing the three parameter family, we need to introduce some notation.  Write $\Sigma^\beta$ as the normal graph of a function $\psi_\beta$ over $\Sigma$.  Recall that $\phi$ is a positive function on $\Sigma$ that solves $L\phi = 1$, and observe that $\psi_\beta /\beta \to \phi$ smoothly as $\beta \to 0$.

The following notation is taken from \cite{KMN}.  
 Fix a point $p\in \Sigma$ and for $x\in \Sigma$ let $r(x)$ be the distance from $x$ to $p$.   Fix a number $R > 0$ to be specified later.  
For each $0 \le t \le R$ define a function $\eta_t$ on $\Sigma$ by 
\[
\eta_t(x) = \begin{cases}
1, &\text{if } r(x) \ge t\\
(1/\log(t))(\log t^2 - \log r(x)), &\text{if } t^2 \le r(x) \le t\\
0 &\text{if } r(x) \le t^2.
\end{cases}
\] 
This function $\eta_t$ will be used to construct the necks.

\begin{defn} Let $X = [-\eps/2,\eps/2]^2 \times [0,R]$ and define 
\[
\Phi \f X \to \mathcal I^{n+1}(M,\Z_2)
\]
as follows.  First, for each $(x,y,t)\in X$ let $S(x,y,t)$ be the union of the graph of $\eta_t \psi_{\eps+x}$ with the graph of $\eta_t  \psi_{-\eps+y}$.  This is a piecewise smooth surface.  Choose a point $q\in \Sigma$ with $r(q) \gg R$.  Then let $\Phi(x,y,t)$ be the open set in $M$ such that $\bd \Phi(x,y,t) = S(x,y,t)$ and $q\notin \Phi(x,y,t)$.  The family $\Phi$ is continuous in the $\mathbf F$ topology. 
\end{defn}
In the next sequence of propositions, we prove the two key properties of the family $\Phi$ outlined in the introduction.

\begin{prop}
The surface $\Sigma^\eps \cup \Sigma^{-\eps}$ is an index two critical point of $A^\eps$.  Moreover, there is a constant $c > 0$ that doesn't depend on $\eps$ such that 
\[
A^\eps(\Phi(x,y,0)) \le A^{\eps}(\Phi(0,0,0)) - c({x^2+y^2})
\]
for all $(x,y,0) \in X$. 
\end{prop} 

\begin{proof}
Since $\Sigma$ is an index one critical point of $A^0$, it follows that $\Sigma^\eps$ is an index one critical point of $A^\eps$.  Likewise $\Sigma^{-\eps}$ is an index one critical point of $A^\eps$ and therefore the union $\Sigma^{\eps}\cup \Sigma^{-\eps}$ is an index two critical point of $A^\eps$. 
Next we study how $A^\eps(\Sigma^t)$ depends on $t$.  Let $L_t$ be the Jacobi operator on $\Sigma^t$.  Since the Jacobi operator on $\Sigma$ is non-degenerate, $L_t$ is also non-degenerate for all sufficiently small $t$.  Moreover, the unique solution $f_t$ to $L_tf_t = 1$ is uniformly positive for $t$ small enough.  Since  
\begin{align*}
\frac{d}{dt} A^\eps(\Sigma^t) = \int_{\Sigma_t} (t - \eps)f_t  \ dv_{\Sigma_t},
\end{align*} 
it follows that there is a constant $c > 0$ such that 
\[
A^\eps(\Sigma^t) \le A^{\eps}(\Sigma^\eps) - c\vert t-\eps\vert^2
\]
for all $0 \le t \le 2\eps$.  The same reasoning applies to $\Sigma^{-\eps}$ and this implies the proposition. 
\end{proof} 

\begin{lem}
Let $c$ be the constant from the previous proposition.  Then for all $\eps$ sufficiently small and all $(x,y,t)\in X$ there is an inequality
\[
\area(\bd \Phi(x,y,t)) \le \area(\bd \Phi (x,y,0)) + \frac{c\eps^2}{2}.
\]
Moreover, $\area(\bd \Phi(x,y,R)) < \area(\bd \Phi(x,y,0))$ for all choices of $x$ and $y$.
\end{lem}

\begin{proof}  
This essentially follows from the proof of Theorem 2.4 in \cite{KMN}.  We include the details for the sake of clarity.
Let $\gamma = \eps + x$ and let $g_{\gamma,t} = \psi_\gamma \eta_t/\gamma$.    
Note that there is a bound $\|g_{\gamma,t}\|_{L^\infty} \le C$ where $C$ is a constant that does not depend on $\gamma$ or $t$.   

For a function $f$ on $B(p,R)\subset\Sigma$, let $S_{f}$ be the normal graph of $f$ over $B(p,R)$.  Proposition 2.5 in \cite{KMN} gives the existence of an $h_0 > 0$ so that for $h \le h_0$ there is an expansion 
\begin{align*}
\area(S_{hg_{\gamma,t}}) &\le \area(B_t) - \area(B_{t^2}) 
\\&\qquad+ \frac{h^2}{2}\int_{B_t \setminus B_{t^2}} (\vert \grad g_{\gamma,t}\vert^2 - g_{\gamma,t}^2(\vert A\vert^2 + \ric(N,N))) \\ 
&\qquad+ Ch^3 \int_{B_t\setminus B_{t^2}} (1+\vert \grad g_{\gamma,t}\vert^2).
\end{align*}
Moreover, the constants $h_0$ and $C$ do not depend on $\eps$ or $t$.  

In particular, for $\gamma < h_0$ we can set $h = \gamma$ in the above expansion to get 
\begin{align*}
\area(S_{\psi_\gamma \eta_t}) &\le \area(B_t) - \area(B_{t^2})\\ &\qquad+ \frac{\gamma^2}{2}\int_{B_t\setminus B_{t^2}} (\vert \grad g_{\gamma,t}\vert^2 - g_{\gamma,t}^2(\vert A\vert^2 + \ric(N,N))) \\ 
&\qquad+ C\gamma^3 \int_{B_t\setminus B_{t^2}} (1+\vert \grad g_{\gamma,t}\vert^2).
\end{align*}
Recall that $\psi_{\gamma}/\gamma \to \phi$ smoothly as $\gamma \to 0$.  Therefore, taking $R$ small enough and $\eps$ small enough, we get that 
\begin{gather*}
\frac{\gamma^2}{2}\left\vert \int_{B_t\setminus B_{t^2}} g_{\gamma}^2(\vert A\vert^2 + \ric(N,N))\right\vert \le \frac {c\gamma^2}{128}.
\end{gather*}
Shrinking $\eps$ further to absorb the $\gamma^3$ terms, this implies that 
\[
\area(S_{\psi_\gamma \eta_t}) \le \area(B_t)
+ \frac{c\gamma^2}{128} + \gamma^2\int_{B_t\setminus B_{t^2}} \vert \grad g_{\gamma,t}\vert^2.
\]
Finally, using the logarithmic cutoff trick as in \cite{KMN} together with the fact that $\psi_\gamma/\gamma \to \phi$ as $\gamma\to 0$, it follows that 
\[
\int_{B_t\setminus B_{t^2}} \vert \grad g_{\gamma,t}\vert^2 \le \frac{c}{128} + \frac{A}{\vert \log t\vert}
\]
where $A$ is a constant that does not depend on $\gamma$ or $t$.  For $R$ small enough, this implies that 
\[
\area(S_{\psi_\gamma\eta_t}) \le \area(B_R) + \frac{c\gamma^2}{32}
\] 
for all $t\in [0,R]$.  

Therefore, letting $\Omega^+ = \{f > 0\} = \cup_{\beta > 0} \Sigma^\beta$, it follows that
\begin{align*}
\area(\bd \Phi(x,y,t) &\cap \Omega^+) -\area(\bd\Phi(x,y,0)\cap \Omega^+) \\
&\le \area(S_{\psi_\gamma\eta_t}) - \area(B_R)(1 - C\eps^2)\\
&\le C\eps^2 \area(B_R) + \frac{c\gamma^2}{32}\\
&\le \frac{c\eps^2}{4}
\end{align*}
provided $R$ is small enough. A similar argument shows that the above inequality is also true with $\Omega^+$ replaced by $\Omega^- = \{f < 0\} = \cup_{\beta < 0} \Sigma^\beta$.  This proves the lemma.
\end{proof} 

\begin{prop}
For every $(x,y,t)\in \bd X$ it holds that
\[
A^\eps(\Phi(x,y,t)) \le A^\eps(\Phi(0,0,0))
\]
with equality if and only if $(x,y,t) = (0,0,0)$. 
\end{prop}

\begin{proof}
Fix a point $(x,y,t) \in \bd X$.  The proposition is clearly true if $t = 0$, and the proposition is true if $t = R$ by the previous lemma.  So assume that $0 < t < R$.  The previous lemma implies that 
\[
\area(\bd \Phi(x,y,t)) \le \area(\bd \Phi(x,y,0)) + \frac{c\eps^2}{2}.
\]
It follows that 
\begin{align*}
A^\eps(\Phi(x,y,t)) &= \area(\bd \Phi(x,y,t)) - \eps \vol(\Phi(x,y,t)) \\
&\le A^\eps(\Phi(x,y,0)) + \frac{c\eps^2}{2}\\
&\le A^{\eps}(\Phi(0,0,0))-\frac{c \eps^2}{2}.
\end{align*}
This proves the proposition.
\end{proof}

\subsection{Non-trivial Width} 

Again fix $(M^{n+1},g)$ and $\Sigma$ satisfying assumption (U-i) and fix a small number $\eps > 0$.  Let $\Pi$ be the $(X,\bd X)$-homotopy class of the map $\Phi$ constructed in the previous section.  Let $\Omega^\eps = \Phi(0,0,0)$ so that $\bd \Omega^\eps = \Sigma^\eps \cup \Sigma^{-\eps}$.  The goal of this section is to prove that the width of $\Pi$ is non-trivial, i.e., to check that 
\[
L^\eps(\Pi) > A^\eps(\Omega^\eps) = \max_{(x,y,t)\in \bd X} A^\eps(\Phi(x,y,t)).
\]
The proof is based on the quantitative minimality results in Appendix \ref{QM}.

\begin{prop} 
\label{nt} 
There are constants $\gamma > 0$ and $\eta > 0$ and $C > 0$ such that the following property holds.  If $\Psi\f X \to \mathcal I_{n+1}(M,\Z_2)$ is an $\mathbf F$-continuous map with 
\[
\sup_{(x,y,t)\in  \bd X} \mathbf F(\Psi(x,y,t), \Phi(x,y,t)) < \eta
\]
then there is a point $(x_0,y_0,t_0)\in X$ such that 
\[
A^\eps(\Psi(x_0,y_0,t_0)) \ge A^\eps(\Omega^\eps) + C\gamma^2.
\]
\end{prop} 

\begin{proof}
Let $\delta > 0$ and $C>0$ be the constants from Theorem \ref{qm} applied to $\Sigma^\eps\cup \Sigma^{-\eps} = \bd \Omega^\eps$. 
Fix some $0 < \gamma < \delta/4$ and then choose a constant $\eta > 0$ to be specified later. 
Consider a map $\Psi$ as in the statement of the proposition.  
If $\eta$ is small enough, it is possible to find a piecewise linear surface $S\subset X$ such that the following properties hold. 
\begin{itemize}
\item 
$\gamma < \fla(\Psi(p), \Omega^\eps) < 2\gamma$ 
for all $p \in S$ \\
\item $\bd S$ is a connected curve in the bottom face of $X$ that encloses $(0,0,0)$. Moreover, $\dist(\bd S, (0,0,0)) > d$ for some positive constant $d$ that doesn't depend on $\Psi$. 
\end{itemize} 
This can be done, for example, by taking a suitable simplicial approximation to the function 
\[
(x,y,t)\in X \mapsto \fla(\Psi(x,y,t), \Omega^\eps).
\]
Note that $A^\eps(\Phi(p)) \le A^\eps(\Omega^\eps) - d_1$ for all $p\in \bd S$.   Here $d_1 > 0$ is a constant that does not depend on $\Psi$. 

Fix a small number $\alpha > 0$. 
By Theorem 3.8 in \cite{MN2}, if $\eta$ is small enough there exists an $\mathbf F$-continuous homotopy 
\[
H\f \bd S\times [0,1] \to \mathcal I_{n+1}(M,\Z_2)
\]
with the properties that 
\begin{itemize} 
\item $H(p,0) = \Psi(p)$ for all $p\in \bd S$, and
\item $H(p,1) = \Phi(p)$ for all $p\in \bd S$, and 
\item $
\mathbf F(H(p,s), \Phi(p)) < \alpha$ 
for all $p\in \bd S$ and all $s\in[0,1]$.  
\end{itemize}
For an appropriate choice of $\alpha$, this ensures that 
\begin{equation}
\label{z}
A^\eps(H(p,s)) \le A^\eps(\Phi(p)) + \frac{d_1}{2} < A^\eps(\Omega^\eps)
\end{equation}
for all $p\in \bd S$ and all $s\in [0,1]$.  

Now let $
S_1 = S \cup_{\bd S} (\bd S \times [0,1]) $
and define a map $\Psi_1 \f S_1 \to \mathcal I_{n+1}(M,\Z_2)$ by 
letting $\Psi_1 = \Psi$ on $S$ and letting $\Psi_1 = H$ on $\bd S\times [0,1]$.  Note that part (ii) of Theorem \ref{qm} applies to the family $\Psi_1$ parameterized by $S_1$.  Therefore, there is some point $q\in S_1$ such that 
\[
A^\eps(\Psi_1(q)) \ge A^{\eps}(\Omega^\eps) + C\fla(\Psi_1(q), \Omega^\eps)^2. 
\]
By (\ref{z}), the point $q = (x_0,y_0,t_0)$ must belong to $S$.  Thus we have exhibited a point $(x_0,y_0,t_0)\in X$ with 
\[
A^\eps(\Psi(x_0,y_0,t_0)) \ge A^\eps(\Omega^\eps) + C\gamma^2,
\]
and the proposition follows.
\end{proof} 

\begin{corollary}
\label{ntw} 
The width of $\Pi$ satisfies $L^\eps(\Pi) > A^\eps(\Omega^\eps)$. 
\end{corollary} 

\begin{proof}
This is an immediate consequence of Proposition \ref{nt}.
\end{proof}

\subsection{Construction of the Doublings} 

Fix $(M^{n+1},g)$ and $\Sigma$ satisfying assumption (U-i).  
In this section $\eps$ will be allowed to vary, and so we write $X^\eps$, $\Phi^\eps$, and $\Pi^\eps$ to emphasize the dependence of these objects on $\eps$. 

\begin{proof} (Theorem \ref{main3}) 
Corollary \ref{ntw} shows that 
\[
L^\eps(\Pi^\eps) > \max_{(x,y,t)\in \bd X^\eps} A^\eps(\Phi^\eps(x,y,t)),
\]
and therefore $\Pi^\eps$ satisfies all the hypotheses of Theorem \ref{mm}.  
Hence min-max produces an almost embedded $\eps$-cmc hypersurface $\Lambda^\eps = \bd \Theta^\eps$ in $M$ with $A^\eps(\Theta^\eps) = L^\eps(\Pi^\eps)$ and $\ind(\Lambda^\eps) \le 3$. 

Observe that 
\[
A^{\eps}(\Phi^\eps(0,0,0)) \le L^\eps(\Pi^\eps) \le \max_{(x,y,t)\in X^\eps} A^\eps(\Phi^\eps(x,y,t)),
\]
and that both bounds for $L^\eps(\Pi^\eps)$ converge to $2\area(\Sigma)$ as $\eps \to 0$.  Therefore the area of $\Lambda^\eps$ converges to $2\area(\Sigma)$ as $\eps \to 0$. 
\end{proof}

\begin{proof} (Theorem \ref{main4}) 
Assume additionally that (U-ii) holds. 
By the compactness theorem for cmc surfaces with bounded area and index, there is an embedded minimal cycle $V$ in $M$ with $\|V\|(M) = 2\area(\Sigma)$  such that $\Lambda^\eps \to V$ as $\eps \to 0$ (up to a subsequence).   Assumption (U-ii) implies that $V = 2\Sigma$.  
\end{proof} 

\subsection{The Non-foliated Case} We close this section with some remarks on the non-foliated case.  Assume that $\Sigma\subset M$ is an index one, non-degenerate minimal hypersurface.  Let $L$ be the Jacobi operator on $\Sigma$ and let $\phi$ be the solution to $L\phi = 1$.  One can show that $\phi$ has at most two nodal domains.  In the case of exactly two nodal domains, $\phi$ changes sign and thus there is no cmc foliation of a neighborhood of $\Sigma$. 

Nevertheless, it is still possible to foliate a neighborhood of $\Sigma$ by surfaces whose mean curvature vectors point away from $\Sigma$.  Let $H$ be the mean curvature operator on $\Sigma$, and let $\zeta > 0$ be the first eigenfunction of $L$.  Then by the implicit function theorem, for every small $\beta > 0$ there is a smooth function $\psi_\beta$ on $\Sigma$ with $H(\psi_\beta) = \beta\zeta$.  The surfaces $\Sigma^\beta = \text{graph}(\psi_\beta)$ foliate a neighborhood of $\Sigma$. 

Let $x$ be a system of coordinates on $\Sigma$ and let $(x,t)$ be Fermi coordinates on a tubular neighborhood of $\Sigma$.  Let $h$ be a smooth, positive function on $M$ such that $h(x,t) = \zeta(x)$ on a tubular neighborhood of $\Sigma$.  Fix some $\eps > 0$ and note that $\Sigma^\eps$ is a critical point of the $A^{\eps h}$ functional defined by 
\[
A^{\eps h}(\Omega) = \area(\bd \Omega) - \eps \int_{\Omega} h.
\]
Using the prescribed mean curvature (pmc) min-max theory of Zhou and Zhu \cite{ZZ1} and the same arguments as above, one can show that there are $\eps h$-pmc surfaces $\Lambda^\eps$ with $\area(\Lambda^\eps) \to 2\area(\Sigma)$.  Generically these are doublings of $\Sigma$.  

\appendix

\section{Quantitative Minimality} 
\label{QM}

This appendix contains a quantitative minimality result for the $A^\eps$ functional.   This result is needed to check that the widths of the min-max families in the paper are non-trivial.    The result is based on the following theorem of Inauen and Marchese \cite{IM}.  

\begin{theorem} (\cite{IM} Theorem 4.3) 
\label{im}
Let $F$ be an elliptic parametric functional on $M^{n+1}$.  Let $\Sigma^n \subset M^{n+1}$ be a smooth, closed, hypersurface which is a non-degenerate, index $k$ critical point for $F$.  Then there are constants $r > 0$, $c>0$, $\delta > 0$, and $C > 0$ and a smooth $k$-parameter family of surfaces 
\[
(\Sigma_v)_{v\in \cl B^k_r}
\]
such that the following properties hold. 
\begin{itemize}
\item[(i)] For every $v\in \cl B^k_r$, the surface $\Sigma_v$ is homologous to $\Sigma$ and satisfies $\fla(\Sigma_v,\Sigma) < \delta$ and $F(\Sigma_v) \le F(\Sigma) -  c\vert v\vert^2.$\\

\item[(ii)] Let $S^k$ be an abstract $k$-manifold with $\bd S^k = \bd \cl B^k_r$.   Then for any continuous family of integral currents
\[
(\tilde \Sigma_v)_{v\in S},
\]
each homologous to $\Sigma$ with $\fla(\tilde \Sigma_v, \Sigma) < \delta$ for all $v\in S$ and $\tilde \Sigma_v = \Sigma_v$ for $v\in \bd S$, it holds that 
\[
\sup_{v\in S} \bigg[F(\tilde \Sigma_v) - C\fla(\tilde \Sigma_v, \Sigma)^2\bigg] \ge F(\Sigma).
\]
\end{itemize}

\begin{rem}
Let $u_1,\hdots,u_k$ be the eigenfunctions for the second variation of $F$ on $\Sigma$ with negative eigenvalues.  Let 
\[
(\psi_v)_{v\in \cl B^k_r}
\]
be a family of smooth functions on $\Sigma$ for which the map 
\[
v\in \cl B^k_r \mapsto \left(\int_\Sigma \psi_v u_1, \hdots, \int_\Sigma \psi_v u_k\right) \in \R^k
\] 
is a diffeomorphism onto a neighborhood of $0$. 
Then by inspecting the proof of Theorem 4.3 in \cite{IM} along with the proofs of Theorem 4 and Theorem 5 in \cite{W2}, one sees that it is possible to take $\Sigma_v = \text{graph}(\psi_v)$ in the above theorem.
\end{rem}  
\end{theorem} 

Unfortunately, Theorem \ref{im} does not apply directly in our setting since the $A^\eps$ functional cannot be written globally as an elliptic parametric functional.  Nevertheless, we have the following.  

\begin{theorem} 
\label{qm} 
Let $\Sigma = \bd \Omega$ be a smooth, closed, hypersurface in $M$ which is a non-degenerate, index $k$ critical point for $A^\eps$.  Then there are constants $r > 0$, $c>0$, $\delta > 0$, and $C > 0$ and a smooth $k$-parameter family of open sets 
\[
(\Omega_v)_{v\in \cl B^k_r}
\]
such that the following properties hold. 
\begin{itemize}
\item[(i)] For every $v\in \cl B^k_r$, the set $\Omega_v$ satisfies $\fla(\Omega_v,\Omega) < \delta$ and $A^\eps(\Omega_v) \le A^\eps(\Omega) -  c\vert v\vert^2.$\\

\item[(ii)] Let $S^k$ be an abstract $k$-manifold with $\bd S^k = \bd \cl B^k_r$.   Then for any $\mathbf F$ continuous family 
\[
(\tilde \Omega_v)_{v\in S}
\]
in $\mathcal I_{n+1}(M,\Z_2)$ with $\fla(\tilde \Omega_v, \Omega) < \delta$ for all $v\in S$ and $\tilde \Omega_v = \Omega_v$ for $v\in \bd S$, there is a point $v \in S$ such that
\[
\sup_{v\in S} \left[A^\eps(\tilde \Omega_v) -C \fla(\tilde \Omega_v, \Omega)^2\right] \ge A^\eps(\Omega).
\]
Moreover, the inequality is strict unless $\tilde \Omega_v = \Omega$.
\end{itemize}
Let $u_1,\hdots,u_k$ be the eigenfunctions for the Jacobi operator on $\Sigma$ with negative eigenvalues.  Let 
\[
(\psi_v)_{v\in \cl B^k_r}
\]
be a family of smooth functions on $\Sigma$ for which the map 
\[
v\in \cl B^k_r \mapsto \left(\int_\Sigma \psi_v u_1, \hdots, \int_\Sigma \psi_v u_k\right) \in \R^k
\] 
is a diffeomorphism onto a neighborhood of $0$. 
Then it is possible to choose $\Omega_v$ above so that $\bd \Omega_v = \text{graph}(\psi_v)$. 
\end{theorem}

 To prove Theorem 5.3, one essentially copies the arguments from \cite{IM} and observes that they continue to hold with $F$ replaced by $A^\eps$.   We include the details for completeness.  

\begin{proof} Let $u_1,\hdots,u_k$ be the eigenfunctions for the Jacobi operator on $\Sigma$ with negative eigenvalues.  Pick a smooth function $\vec f\f M \to \R^k$ such that 
\[
\vec f(x) = 0, \quad \text{and} \quad \grad \vec f(x) = (u_1(x),\hdots,u_k(x)) 
\]
for all $x\in \Sigma$.  Let $K$ be a very large constant and define 
\[
G(\Theta) = A^\eps(\Theta) + K \left\|\int \vec f \, d\|\bd \Theta\| \right\|^2
\]
for $\Theta\in \mathcal I_{n+1}(M,\Z_2)$.  It follows from \cite{W2} that the functional $G$ is lower-semicontinuous with respect to flat convergence, and $\Sigma = \bd \Omega$ is a strictly stable critical point of $G$. 

\begin{lem} 
\label{c}
There is some $\delta > 0$ such that $G(\Omega) < G(\Theta)$ for all $\Theta \neq \Omega$ with $\fla(\Theta,\Omega) < \delta$. 
\end{lem} 

\begin{proof} Suppose for contradiction that this is not the case.  Then there are sets $\Omega_i \neq \Omega$ with $\fla(\Omega_i,\Omega)\to 0$ and $G(\Omega_i) \le G(\Omega)$.  Define 
\[
G_i(\Theta) = G(\Theta) + \lambda \vert \fla(\Theta,\Omega) -\fla(\Omega_i,\Omega)\vert,
\]
where $\lambda>0$ is a constant to be specified later.  Let $\Omega_i'$ be a minimizer of $G_i$.  Passing to a subsequence, $\Omega_i' \to \Omega'$ in the flat topology.  The proof of Lemma 3.3 in \cite{IM} applies verbatim to show that $\Omega'$ minimizes 
\[
G_0(\Theta) = G(\Theta) + \lambda \vert \fla(\Theta,\Omega)\vert
\]
over all $\Theta \in \mathcal I_{n+1}(M,\Z_2)$. 

Next one verifies the analog of Lemma 3.5 in \cite{IM}. 

\begin{lem}
There are constants $\delta > 0$ and $C>0$ such that 
\[
G(\Omega) - G(\Theta) \le C\fla(\Omega,\Theta)
\]
for all $\Theta \in \mathcal I_{n+1}(M,\Z_2)$. 
\end{lem}

\begin{proof}
Note that 
\begin{align*}
G(\Omega) &- G(\Theta) \\
& = [\area(\bd \Omega) - \area(\Theta)] - \eps[\vol(\Omega) - \vol(\Theta)] - K \left\|\int \vec f \, d\|\bd \Theta\| \right\|^2\\
&\le [\area(\bd \Omega) - \area(\Theta)] - \eps[\vol(\Omega) - \vol(\Theta)]\\
&\le [\area(\bd \Omega) - \area(\Theta)] + \eps\fla(\Omega,\Theta). \phantom{\int}
\end{align*}
By Lemma 3.5 in \cite{IM}, there is a constant $C$ such that $\area(\bd \Omega) - \area(\bd \Theta) \le C\fla(\Omega,\Theta)$, and the lemma follows.
\end{proof}

The proof of Lemma 3.6 in \cite{IM} now applies verbatim to show that $\Omega$ is the only minimizer of $G_0$.  Thus the minimizers $\Omega_i'$ converge to $\Omega$ in the flat topology.  We claim that in fact $\Omega_i' \to \Omega$ in the $\mathbf F$-topology. Indeed, since $\Omega_i'$ minimizes $G_i$, there is an inequality
\begin{align*}
G(\Omega_i') + \lambda\vert \fla(\Omega_i',\Omega)-\fla(\Omega_i,\Omega)\vert \le G_i(\Omega_i) = G(\Omega_i) \le G(\Omega).
\end{align*} 
This implies that
\[
\area(\bd \Omega_i') - \eps\vol(\Omega_i') \le \area(\bd \Omega) - \eps\vol(\Omega),
\]
and it follows that 
\[
\limsup \area(\bd \Omega_i') \le \area(\bd \Omega)
\]
since $\vol(\Omega_i')\to \vol(\Omega)$. This proves the $\mathbf F$-convergence. 

Now observe that the varifolds $\vert \Omega_i'\vert$ have uniformly bounded first variation.  This implies that they satisfy a monotonicity formula with uniform constants.  Since $\Omega_i' \to \Omega$ in the $\mathbf F$-topology, it follows that $\bd \Omega_i'$ is eventually contained in a tubular neighborhood of $\Sigma$. According to White \cite{W2}, this implies that $G(\Omega_i') > G(\Omega)$, and this is a contradiction.  This establishes Lemma \ref{c}.  
\end{proof} 

\begin{lem}
\label{c2}
There are constants $\delta > 0$ and $C>0$ such that 
\[
G(\Omega) \le G(\Theta) + C\fla(\Omega,\Theta)^2
\]
for all $\Theta\in \mathcal I_{n+1}(M,\Z_2)$ with $\fla(\Omega,\Theta) < \delta$.
\end{lem} 

\begin{proof} 
Think of $C > 0$ as a fixed constant to be chosen later.   Suppose for contradiction that the claim fails.  Then there are sets $\Omega_i\neq \Omega$ with $\fla(\Omega_i,\Omega)\to 0$ and 
\[
G(\Omega_i) + C\fla(\Omega_i,\Omega)^2 \le G(\Omega) 
\]
Define 
\[
H_i(\Theta) = G(\Theta) + \lambda[\fla(\Theta,\Omega) - \fla(\Omega_i,\Omega)]^2
\]
where $\lambda > 0$ is a constant to be specified later.  Let $\Omega_i'$ be a minimizer of $H_i$.  Passing to a subsequence, $\Omega_i' \to \Omega'$ in the flat topology. The proof of Lemma 4.1 in \cite{IM} applies verbatim to show that $\Omega'$ minimizes 
\[
H_0(\Theta) = G(\Theta) + \lambda \fla(\Theta,\Omega)^2
\]
over all $\Theta\in \mathcal I_{n+1}(M,\Z_2)$. 

We claim that $\Omega$ is the unique minimizer of $H_0$ provided $\lambda$ is large enough.  Suppose for contradiction that there is some $\Omega_1 \neq \Omega$ with $H_0(\Omega_1) \le H_0(\Omega)$.  Then 
\begin{align*}
G(\Omega_1) + \lambda \fla(\Omega_1,\Omega)^2 \le A^\eps(\Omega)
\end{align*}
which implies that 
\[
\fla(\Omega_1,\Omega)^2 \le \frac{A^\eps(\Omega) - A^\eps(\Omega_1)}{\lambda} \le \frac{A^\eps(\Omega) + \eps \vol(M)}{\lambda}. 
\]
In particular, if $\lambda$ is large enough then Claim \ref{c} applies to $\Omega_1$ and so $G(\Omega_1) > G(\Omega)$. This is a contradiction. 

Since $\Omega$ is the unique minimizer of $H_0$, it follows that $\Omega_i' \to \Omega$ in the flat topology.  The same argument as above shows that this convergence is actually in the $\mathbf F$-topology.  Again the varifolds $\vert \Omega_i'\vert$ satisfy a monotonicity formula with uniform constants and hence are eventually contained in a tubular neighborhood of $\Sigma$.  This contradicts Theorem 1.1 in \cite{IM} since the $A^\eps$ functional can locally be written as an elliptic parametric functional.  (This is because the volume form $\omega$ on $M$ is exact in a tubular neighborhood of $\Sigma$.)  This establishes Lemma \ref{c2}.
\end{proof}

Finally Theorem \ref{qm} follows from Lemma \ref{c2} as explained in \cite{W2}. 
\end{proof} 

Note that Theorem \ref{qm} has the following corollary. 

\begin{corollary}
\label{qmss}
Let $\Sigma = \bd \Omega$ be a smooth, closed, $\eps$-cmc in $M$ which is strictly stable for $A^\eps$.  Then there are constants $\delta > 0$ and $C> 0$ such that every $\tilde \Omega \in \mathcal I_{n+1}(M,\Z_2)$ with $\fla(\tilde \Omega,\Omega) < \delta$ satisfies 
$
A^\eps(\tilde \Omega) \ge A^\eps(\Omega) + C\fla(\tilde\Omega,\Omega)^2.
$
\end{corollary}

\section{generic metrics}
\label{gm} 

It is natural to ask whether assumption (U-ii) poses a significant restriction to the applicability of Theorem \ref{main4}.  The following proposition addresses this question.  It shows that assumption (U-ii) holds for a generic set of metrics $g$ on $M$. 

\begin{prop}
\label{g}
Let $M$ be a closed manifold.  There is a (Baire) generic set $\mathcal G$ of smooth metrics on $M$ with the following property:  if $g\in \mathcal G$ then for any closed, connected, embedded minimal hypersurface $\Sigma$ in $(M,g)$ the varifold $2\Sigma$ is the only embedded minimal cycle in $(M,g)$ with area $2\area(\Sigma)$. 
\end{prop}

Proposition \ref{g} is a corollary of the following result of Marques and Neves \cite{MN2}.  Given a metric $g$ on $M$ and $C >0$ and $I\in \N$, let $\mathcal M_{C,I}(g)$ denote the collection of all closed, connected, embedded minimal hypersurfaces in $(M,g)$ with area at most $C$ and index at most $I$.

\begin{prop}
\label{a} 
(\cite{MN2} Proposition 8.6) Let $g$ be a bumpy metric on $M$, and fix $C > 0$ and $I \in \N$.   There exist metrics $\tilde g$ arbitrarily close to $g$ in the smooth topology such that the following properties hold. 
\begin{itemize}
\item[(i)] The set $\mathcal M_{C,I}(\tilde g) = \{\Sigma_1,\hdots,\Sigma_N\}$ is finite and every surface in $\mathcal M_{C,I}(\tilde g)$ is non-degenerate.\\
\item[(ii)] The areas $\area_{\tilde g}(\Sigma_1)$, $\hdots$, $\area_{\tilde g}(\Sigma_N)$ are linearly independent over $\Q$. 
\end{itemize}
\end{prop} 

\begin{rem}
Note that property (ii) above immediately implies the following weaker property.
\begin{itemize}
\item[(iii)] Let $A = a_1\area_{\tilde g}(\Sigma_1) + \hdots + a_N \area_{\tilde g}(\Sigma_N)$ for some integers $a_i \ge 0$.  If $A = 2\area_{\tilde g}(\Sigma_i)$ for some $i$ then $a_i = 2$ and all the other $a_j$'s are zero.  
\end{itemize}
\end{rem}

\begin{proof}
(Proposition \ref{g}) Given $C > 0$ and $I\in \N$, let $\mathcal G_{C,I}$ be the collection of all metrics $g$ on $M$ for which properties (i) and (iii) above hold (with $g$ in place of $\tilde g$).  We claim that $\mathcal G_{C,I}$ is open and dense in the set of all smooth metrics on $M$. 

First we show that $\mathcal G_{C,I}$ is open.  Fix some $g\in \mathcal G_{C,I}$ and write 
\[
\mathcal M_{C,I}(g) = \{\Sigma_1,\hdots,\Sigma_N\}.
\]
Since every surface in $\mathcal M_{C,I}(g)$ is non-degenerate, there is a neighborhood $U$ of $g$ such that for any $\tilde g \in U$ and any $i=1,\hdots,N$ there is a unique minimal surface $\Sigma_i(\tilde g)$ in $(M,\tilde g)$
that is smoothly close to $\Sigma_i$.  Moreover, these surfaces $\Sigma_i(\tilde g)$ are all non-degenerate.  By Sharp's compactness theorem \cite{S}, it follows that there is a potentially smaller neighborhood $U_1$ of $g$ such that 
\[
\mathcal M_{C,I}(\tilde g) \subseteq \{\Sigma_1(\tilde g),\hdots,\Sigma_N(\tilde g)\}
\]
for all $\tilde g\in U_1$.  Taking an even smaller neighborhood $U_2$ of $g$, it is then possible to ensure that condition (iii) holds for all $\tilde g \in U_2$. 

Next we show that $\mathcal G_{C,I}$ is dense.  Consider any metric $g$ on $M$.  Since bumpy metrics are dense, there is a bumpy metric $g_1$ on $M$ arbitrarily close to $g$.  Applying Proposition 3.2 to $g_1$ then yields $g_2 \in \mathcal G_{C,I}$ that is arbitrarily close to $g_1$.   Thus there is a metric $g_2\in \mathcal G_{C,I}$ arbitrarily close to $g$ in the smooth topology. 

To conclude the proof, take sequences $C_n \to \infty$ and $I_n\to \infty$ and define 
\[
\mathcal G = \bigcap_n \mathcal G_{C_n,I_n}. 
\]
Then $\mathcal G$ is Baire generic, and every metric $g\in \mathcal G$ satisfies the conclusion of Proposition \ref{g}.
\end{proof}

\bibliographystyle{plain}
\bibliography{cmc-doubling}

\end{document}